\documentclass[12pt,a4paper,bibliography=totocnumbered,listof=totocnumbered]{scrartcl}
\DeclareOldFontCommand{\rm}{\normalfont\rmfamily}{\mathrm}
\DeclareOldFontCommand{\sf}{\normalfont\sffamily}{\mathsf}
\DeclareOldFontCommand{\tt}{\normalfont\ttfamily}{\mathtt}
\DeclareOldFontCommand{\bf}{\normalfont\bfseries}{\mathbf}
\DeclareOldFontCommand{\it}{\normalfont\itshape}{\mathit}
\DeclareOldFontCommand{\sl}{\normalfont\slshape}{\@nomath\sl}
\DeclareOldFontCommand{\sc}{\normalfont\scshape}{\@nomath\sc}
\usepackage[ngerman, english]{babel}
\usepackage[utf8]{inputenc}
\usepackage{amsmath}
\usepackage{amsthm}
\usepackage{mathtools}\usepackage{amsfonts}
\usepackage{amssymb}
\usepackage{graphicx}
\usepackage{fancyhdr}
\usepackage{tabularx}
\usepackage{geometry}
\usepackage{setspace}
\usepackage[right]{eurosym}
\usepackage[printonlyused]{acronym}
\usepackage{subfig}
\usepackage{floatflt}
\usepackage{bbm}
\usepackage[usenames,dvipsnames]{color}
\usepackage{colortbl}
\usepackage{array}
\usepackage{titlesec}
\usepackage{parskip}
\usepackage[right]{eurosym}
\usepackage[ngerman]{babel}
\usepackage[subfigure,titles]{tocloft}
\usepackage[pdfpagelabels=true, colorlinks, linkcolor = blue , citecolor = blue, filecolor = blue, urlcolor = blue]{hyperref}

\newcommand{\supp}{\mathop{\mathrm{supp}}}
\usepackage{listings}
\usepackage[autostyle=true,german=quotes]{csquotes}
\makeatletter
\def\l@lstlisting#1#2{\@dottedtocline{1}{0em}{1em}{\hspace{1,5em} Lst. #1}{#2}}
\makeatother

\usepackage{cite}

\newtheoremstyle{normal}
{10pt}
{10pt}
{}
{}
{\bfseries}
{}
{0em}
{\bfseries{\thmname{#1}\thmnumber{ #2}\thmnote{\hspace{0em}(#3)\newline}}}

\newtheoremstyle{standard}  
  {10pt}   
  {}   
  {\itshape}  
  {}       
  {\bfseries} 
  {:}         
  {0.2cm}  
  {\bfseries{\thmname{#1}\thmnumber{ #2}\thmnote{ \hspace{0em}(#3)}}}          

\newtheoremstyle{mittitel}  
  {10pt}   
  {}   
  {\itshape}  
  {}       
  {\bfseries} 
  {:}         
  {0.2cm}  
  {\bfseries{\thmname{#1}\thmnumber{ #2}\thmnote{ \hspace{0em}(#3)\newline}}}          

\DeclareMathOperator{\sgn}{sgn}

\DeclareMathOperator{\spz}{sp_\textit{z}}
\DeclareMathOperator{\sqz}{sq_\textit{z}}
\DeclareMathOperator{\cpz}{cp_\textit{z}}
\DeclareMathOperator{\cqz}{cq_\textit{z}}

\DeclareMathOperator{\sq}{sq}
\DeclareMathOperator{\cp}{cp}

\DeclareMathOperator{\spzm}{sp_{\textit{z}, \textit{m}}}
\DeclareMathOperator{\sqzm}{sq_{\textit{z}, \textit{m}}}
\DeclareMathOperator{\cpzm}{cp_{\textit{z}, \textit{m}}}
\DeclareMathOperator{\cqzm}{cq_{\textit{z}, \textit{m}}}

\DeclareMathOperator{\sinp}{sinp}
\DeclareMathOperator{\sinq}{sinq}
\DeclareMathOperator{\cosp}{cosp}
\DeclareMathOperator{\cosq}{cosq}

\DeclareMathOperator{\s}{\mathcal{S}}

\DeclareMathOperator{\sinpm}{sinp_\textit{m}}
\DeclareMathOperator{\sinqm}{sinq_\textit{m}}
\DeclareMathOperator{\cospm}{cosp_\textit{m}}
\DeclareMathOperator{\cosqm}{cosq_\textit{m}}

\DeclareMathOperator{\D}{\mathcal{D}}
\setkomafont{section}{10pt}

\begin{document}

\newtheorem{thm}{Theorem}[section]
\newtheorem{satz}[thm]{Satz} 
\newtheorem{prop}[thm]{Proposition} 
\newtheorem{lem}[thm]{Lemma}
\newtheorem{kor}[thm]{Korollar} 
\newtheorem{defi}[thm]{Definition} 
\newtheorem{konst}[thm]{Konstruktion}
\newtheorem{bemerkung}[thm]{Remark}
\newtheorem{beispiel}[thm]{Example}

\onehalfspacing

\renewcommand{\cfttabpresnum}{Tab. }
\renewcommand{\cftfigpresnum}{Abb. }
\settowidth{\cfttabnumwidth}{Abb. 10\quad}
\settowidth{\cftfignumwidth}{Abb. 10\quad}
\renewcommand{\arraystretch}{1.5}

\begin{center}
\textbf{\Large{Eigenvalue Approximation for Krein-Feller-Operators}} \\[8pt]
Uta Freiberg\footnote{ Institute of Stochastics and Applications, University of Stuttgart, Pfaffenwaldring 57, 70569 Stuttgart, Germany, e-mail: Lenon.Minorics@mathematik.uni-stuttgart.de}, Lenon Minorics\footnote{ Institute of Stochastics and Applications, University of Stuttgart, Pfaffenwaldring 57, 70569 Stuttgart, Germany, e-mail: Uta.Freiberg@mathematik.uni-stuttgart.de}
\end{center}

\titleformat*{\section}{\large\bfseries}
\titlespacing{\section}{0pt}{12pt plus 4pt minus 2pt}{-6pt plus 2pt minus 2pt}
\textbf{Abstract:}
We study the limiting behavior of the eigenvalues of Krein-Feller-Operators with respect to weakly convergent probability measures.
Therefore, we give a representation of the eigenvalues as zeros of measure theoretic sine functions. Further, we make a proposition about the limiting behavior of the previously determined eigenfunctions. With the main results we finally determine the speed of convergence of eigenvalues and -functions for sequences which converge to invariant measures on the Cantor set.

\section{Introduction}
Let $\mu$ be a non-atomic Borel probability measure on $[0,1]$ and 
\begin{align*}
\D_1^\mu \coloneqq \bigg\{f: [0,1] \longrightarrow \mathbb{R}: ~ \exists ~ f^\mu \in ~&L_2([0,1],\mu): \\& f(x) = f(0) + \int_0^x f^{\mu}(y) ~ d \mu (y), ~~~ x \in [0,1] \bigg\}.
\end{align*}
$f^\mu$ is called $\boldsymbol \mu$\textbf{-Derivative} of $f$. Further, define
\begin{align*}
\D \coloneqq \bigg\{ f \in C^1([0,1]) :& ~ \exists ~ (f')^\mu \in L_2([0,1], \mu): \\& f'(x) = f'(0) + \int_0^x (f')^\mu(y) ~  d \mu (y), ~~~ x \in [0,1]  \bigg\}.
\end{align*}
Then, the Krein-Feller-Operator w.r.t. $\mu$ is given as
\begin{align*}
\frac{d}{d \mu} \frac{d}{d x}: \D &\longrightarrow  L_2([0,1], \mu) \\
f &\mapsto ~~ (f')^\mu.
\end{align*}
Analytic properties of Krein-Feller-Operators are developed in \cite{Fre03}. Many papers deal with this operator and with the resulting eigenvalue problem, see for example Feller \cite{Fel57},  Freiberg et al. \cite{Fre03, Fre05, Fre08, Fre11, Fre19, Fre04, Fre18, Fre02}, Fujita \cite{Fuj87}, Minorics \cite{Min17, Min18}, Ngai et al. \cite{Nga11, Nga18, Bir03, Che10, Hu06} and for higher dimensional generalizations Freiberg and Seifert \cite{Fre15} and Solomyak et al. \cite{Sol95, Nai95}.
\\
In this paper we consider the corresponding eigenvalue problem 
\begin{align}
\frac{d}{d \mu} \frac{d}{d x} f = - \lambda f \label{eigenwertproblem}
\end{align}
with Dirichlet or Neumann boundary conditions. In \cite[Theorem 1]{Kue80} it is shown, that the eigenvalues of \eqref{eigenwertproblem} with Dirichlet or Neumann boundary conditions are countable infinite, have no finite accumulation points and multiplicity one. Moreover, if the sequence of Neumann eigenvalues is given by $(\lambda_{N,m})_{m \in \mathbb{N}_0}$ and the sequence of Dirichlet eigenvalues by $(\lambda_{D,m})_{m \in \mathbb{N}}$, then
\begin{align*}
0 = \lambda_{N,0} < \lambda_{N,1} < \lambda_{N,2} < ... ~~~ \text{ and } ~~~ 0 < \lambda_{D,1} < \lambda_{D,2} < ..., 
\end{align*}
where $\lambda$ is a Neumann and Dirichlet eigenvalue, if it solves \eqref{eigenwertproblem} with Neumann and Dirichlet boundary conditions, respectively.
In \cite[Chapter 4]{Arz15} a concept of measure theoretic trigonometric functions is developed, whereby the zeros of measure theoretic sine functions are the eigenvalues of \eqref{eigenwertproblem} with Dirichlet or Neumann boundary conditions. \\
We consider sequences of probability measures $(\mu_n)_n$ those distribution functions $(F_n)_n$ converge uniformly to the distribution function $F$ of some Borel probability measure $\mu$ and show that the corresponding eigenvalues satisfy
\begin{align*}
\begin{split}
|\lambda_{N,m} - \lambda_{N,m,n}| &\le c(m) \lVert F - F_n\rVert_\infty \\ |\lambda_{D,m} - \lambda_{D,m,n}|&\le  c(m) \lVert F - F_n\rVert_\infty
\end{split}~~~~~~~ n \ge n_0(m),
\end{align*}
where $(\lambda_{N,m,n})_{m \in \mathbb{N}_0}$ denotes the sequence of Neumann and $(\lambda_{D,m,n})_{m \in \mathbb{N}}$ the sequence of Dirichlet eigenvalues of the Krein-Feller-Operator w.r.t. $\mu_n$ , respectively.

As an example, we then consider Krein-Feller-Operators w.r.t. $\mu^w$, where $\mu^w$ is given as the unique invariant Borel probability measure to the IFS $\s = (S_1,S_2)$, $S_1(x) = \frac{1}{3}x$, $S_2(x) = \frac{1}{3}x + \frac{2}{3}$, $x \in [0,1]$ and weight vector $w = (w_1,w_2)$, $w_1 \in (0,1)$, $w_2 = 1-w_1$. Therefore $\mu^w$ is singular w.r.t. the one-dimensional Lebesgue measure. The concept of invariant measures is developed in \cite{Hut81}. We construct a sequence of non-atomic Borel probability measures $(\mu_n^w)_{n \in \mathbb{N}_0}$, $\mu_n^w \rightharpoonup \mu^w$ and get
\begin{align*}
\begin{split}
|\lambda_{N,m} - \lambda_{N,m,n}| &\le c(m) (w_1 \vee w_2)^n \\ |\lambda_{D,m} - \lambda_{D,m,n}|&\le  c(m) (w_1 \vee w_2)^n
\end{split}~~~~~~~ n \ge n_0(m).
\end{align*}
For a treatment of the classical theory of boundary problems on the real line see e.g. Atkinson \cite{Atk}.
\newpage
\section{Measure theoretic trigonometric functions}
Let $\mu$ be a non-atomic Borel probability measure on $[0,1]$.
\begin{defi} \label{trig}
Let $x \in [0,1]$, $z \in \mathbb{R}$ and $p_0(x) \coloneqq q_0(x) \coloneqq 1$. For $n \in \mathbb{N}$ let
\begin{align*}
p^\mu_n(x) \coloneqq p_n (x) \coloneqq
\begin{cases}
\int_0^x p_{n-1} (t) \, d \mu (t), & \text{if } n \text{ is odd} \\
\int_0^x p_{n-1} (t) \, d t, & \text{if } n \text{ is even}
\end{cases}
\\\\
q^\mu_n(x) \coloneqq q_n  (x) \coloneqq
\begin{cases}
\int_0^x q_{n-1} (t) \, d t,& \text{if } n \text{ is odd} \\
\int_0^x q_{n-1} (t) \, d\mu (t), & \text{if } n \text{ is even}
\end{cases}
\end{align*}
and
\begin{align*}
\spz (x) &\coloneqq \sum_{n=0}^\infty (-1)^n \, z^{2n+1} p_{2n+1}(x), ~~~~~~~~~~~~ &&\sqz (x) \coloneqq \sum_{n=0}^\infty (-1)^n \, z^{2n+1} q_{2n+1}(x), \\[8pt]
\cpz (x) &\coloneqq \sum_{n=0}^\infty (-1)^n\, z^{2n} p_{2n}(x), ~~~~~~~~~~~ &&\cqz (x)  \coloneqq \sum_{n=0}^\infty (-1)^n\, z^{2n} q_{2n} (x).
\end{align*}
\end{defi}
\begin{lem} \label{absch}
For all $x \in [0,1]$, $z \in \mathbb{R}$ and $n \in \mathbb{N}_0$ holds
\begin{align*}
 p_{2n+1}(x) &\le \frac{1}{n!} (q_2(x))^n , ~~~~~~~~~~~~~~~~~~ p_{2n}(x) \le \frac{1}{n!} (p_2(x))^n, \\[8pt]
 q_{2n+1}(x) &\le \frac{1}{n!} (p_2(x))^n , ~~~~~~~~~~~~~~~~~~ q_{2n}(x) \le \frac{1}{n!} (q_2(x))^n.
\end{align*}
\end{lem}
\begin{proof}
\cite[Lemma 2.4]{Fre04}.
\end{proof}
\begin{lem}
For fixed $z \in \mathbb{R}$ the series in Definition \ref{trig} converge uniformly absolutely on $[0,1]$ and
\begin{align*}
\frac{d}{d \mu(x)} \spz (x) &= z \cpz (x),   &&  \frac{d}{d x}\sqz = z \cqz(x), \\[10pt]
\frac{d}{d \mu(x)} \cqz(x) &= - z \sqz(x),  && \frac{d}{d x}\cpz = -z \spz(x).   
\end{align*}
\end{lem}
\begin{proof}
\cite[Lemma 3.6]{Arz15}.
\end{proof}
\newpage
\begin{thm}
\begin{enumerate}
\item[(i)]
The Neumann eigenvalues $\lambda_{N,m}$, $m \in \mathbb{N}_0$ are the squares of the positive zeros of the function $\sinp(z) \coloneqq \spz(1)$, $z \in \mathbb{R}$. Up to a multiplicative constant, the corresponding eigenfunctions are given by
\begin{align*}
f_{N,m}(x) \coloneqq \cp_{(\lambda_{N,m})^{1/2}}(x), ~~~~ x \in [0,1].
\end{align*}
\item[(ii)]
The Dirichlet eigenvalues $\lambda_{D,m}$, $m \in \mathbb{N}$ are the squares of the non-negative zeros of the function $\sinq(z) \coloneqq \sqz(1)$, $z \in \mathbb{R}$. Up to a multiplicative constant, the corresponding eigenfunctions are given by
\begin{align*}
f_{D,m}(x) \coloneqq \sq_{(\lambda_{D,m})^{1/2}}(x), ~~~~ x \in [0,1].
\end{align*}
\end{enumerate}
\end{thm}
\begin{proof}
\cite[Proposition 3.8]{Arz15}
\end{proof}
To prove the following statements, we need the multiplication formula 
\begin{align}
\bigg( \sum_{j=0}^\infty a_{2j}  \bigg) \cdot \bigg( \sum_{k=0}^\infty b_{2k} \bigg) = \sum_{n=0}^\infty \sum_{k=0}^n a_{2k}\, b_{2n-2k}, \label{rechenregel}
\end{align}
where $(a_n)_{n \in \mathbb{N}_0}$ and $(b_n)_{n \in \mathbb{N}_0}$ are absolutely summable sequences.
\begin{lem} \label{nullsummen}
For all $m \in \mathbb{N}$ holds
\begin{align*}
\sum_{n=0}^\infty (-1)^n \lambda^n_{N,m} \sum_{k=0}^n  \, p_{2k}\, p_{2n-2k+1} = 0,~~~~ \sum_{n=0}^\infty (-1)^n \lambda^n_{N,m} \sum_{k=0}^n 2k \, p_{2k}\, p_{2n-2k+1} = 0,
\end{align*}
where $p_n \coloneqq p_n(1)$.
\end{lem}
\begin{proof}
For all $z \in \mathbb{R}$ follows with \eqref{rechenregel}
\begin{align*}
 \cosp(z) \cdot \sinp(z) &= \bigg( \sum_{j=0}^\infty(-1)^j\, z^{2j} \, p_{2j} \bigg) \cdot  \bigg( \sum_{k=0}^\infty (-1)^k \, z^{2k+1} \, p_{2k+1}\bigg )  \\
&= \sum_{n=0}^\infty (-1)^n\, z^{2n+1} \sum_{k=0}^n p_{2k}\, p_{2n-2k+1}.
\end{align*}
Let $m \in \mathbb{N}$, $z_m \coloneqq \sqrt{\lambda_{N,m}} \neq 0$. Then $\sinp(z_m) = 0$ and thus
\begin{align*}
\sum_{n=0}^\infty (-1)^n \lambda^n_{N,m} \sum_{k=0}^n p_{2k}\, p_{2n-2k+1} = \sum_{n=0}^\infty (-1)^n\, z_m^{2n} \sum_{k=0}^n p_{2k}\, p_{2n-2k+1} = 0.
\end{align*}
Analogously we get
\begin{align*}
 \cosp ' (z) \cdot \sinp (z) &=  \bigg( \sum_{j=0}^\infty (-1)^j\hspace{1mm} 2j\hspace{1mm} z^{2j-1} \hspace{1mm} p_{2j}\bigg) \cdot  \bigg( \sum_{k=0}^\infty (-1)^k\, z^{2k+1}\, p_{2k+1}\bigg ) \\
&= \sum_{n=0}^\infty (-1)^n\, z^{2n} \sum_{k=0}^n 2k \hspace{1mm} p_{2k}\, p_{2n-2k+1}
\end{align*}
and thus
\begin{align*}
\sum_{n=0}^\infty (-1)^n \lambda^n_{N,m} \sum_{k=0}^n 2k \hspace{1mm}p_{2k}\, p_{2n-2k+1} = \sum_{n=0}^\infty (-1)^n\, z_m^{2n} \sum_{k=0}^n 2k \hspace{1mm} p_{2k}\, p_{2n-2k+1} = 0.
\end{align*}
\end{proof}
\begin{lem} \label{ungleichnull}
For all $m \in \mathbb{N}$ holds
\begin{align*}
\parallel f_{N,m}\parallel_{L_2([0,1],\mu)}^2 = \frac{1}{2} \cosp (\sqrt{\lambda_{N,m}}) \cdot \sinp '(\sqrt{\lambda_{N,m}}).
\end{align*}
\end{lem}
\begin{proof}
Let $m \in \mathbb{N}$ and $z_m \coloneqq \sqrt{\lambda_{N,m}} \neq 0$. Then 
\begin{align*}
\sinp'(z_m) = \sum_{k=0}^\infty (-1)^k \ 2k \ z_m^{2k} \ p_{2k+1}
\end{align*}
and hence
\begin{align*}
\cosp (z_m) \cdot \sinp' (z_m) &= \bigg( \sum_{j=0}^\infty (-1)^j \, z_m^{2j} \hspace{1mm} p_{2j}\bigg) \cdot  \bigg( \sum_{k=0}^\infty (-1)^k \ 2k  \hspace{1mm} z_m^{2k} \ p_{2k+1}\bigg ) \\
&= \sum_{n=0}^\infty (-1)^n \, z_m^{2n} \sum_{k=0}^n (2n-2k)\, p_{2k}\, p_{2n-2k+1} \\
&= \sum_{n=0}^\infty (-1)^n \, z_m^{2n} \sum_{k=0}^n 2n\, p_{2k}\, p_{2n-2k+1},
\end{align*}
whereby the last equality follows from Lemma \ref{nullsummen}. In \cite{Arz15} Corollary 4.3 the formula 
\begin{align*}
\parallel f_{N,m}\parallel_{L_2([0,1],\mu)}^2 = \sum_{n=0}^\infty (-1)^n \lambda^n_{N,m} \sum_{k=0}^n (n+1-2k)\, p_{2k}\, p_{2n-2k+1}
\end{align*}
is shown. Together with Lemma \ref{nullsummen} we get
\begin{align*}
\parallel f_{N,m}\parallel_{L_2([0,1],\mu)}^2 = \sum_{n=0}^\infty (-1)^n \lambda^n_{N,m} \sum_{k=0}^n n\, p_{2k}\, p_{2n-2k+1}.
\end{align*}
Thus the statement follows.
\end{proof}
\begin{prop}
Let $z \in (0, \infty).$ If $z$ is a zero of $\sinp$, then $z$ is no local extremum of $\sinp$.
\end{prop}
\begin{proof}
If $z \in (0,\infty)$ is a local extremum of $\sinp$, then $\sinp'(z) = 0$. Because $\parallel f_{N,m}\parallel_{L_2(\mu)} \neq 0$, the statement follows with Lemma \ref{ungleichnull}.
\end{proof}
Analogously we get the following proposition.
\begin{prop}
Let $z \in (0, \infty).$ If $z$ is a zero of $\sinq$, then $z$ is no local extremum of $\sinq$.
\end{prop}
\section{Eigenvalue approximation}
The main results of this paper are included in this section. Therefore, let $\mu$ be a finite non-atomic Borel probability measure on $[0,1]$ with distribution function $F$. Further, let $(\mu_n)_n$ be a sequence of non-atomic Borel probability measures on $[0,1]$ with distribution functions $(F_n)_n$ such that $F_n$ converges uniformly to $F$.

Before stating the main results, we need some estimates to get the speed of convergence of the measure theoretic trigonometric functions. Therefore, we denote $p_n^\mu$ and $q_n^\mu$ by $p_n$ and $q_n$ respectively and $p_n^{\mu_{m}}$ and $q_n^{\mu_m}$ by $p_{n,m}$ and $q_{n,m}$ respectively.

\begin{lem} \label{pq absch}
For all $x \in [0,1]$ and all $m,n \in \mathbb{N}$ holds 
\begin{align*}
|q_{2n}(x)-q_{2n,m}(x)| &\le 2 \, \frac{\lVert F-F_m \rVert_\infty\, x^n}{ (n-1)!}  , ~~~~~~~ |p_{2n}(x)-p_{2n,m}(x)| \le   2 \, \frac{\lVert F-F_m \rVert_\infty\, x^n}{ (n-1)!},\\[5pt]
|q_{2n+1}(x)-q_{2n+1,m}(x)| &\le 2 \, \frac{\lVert F-F_m \rVert_\infty\, x^n}{ (n-1)!},~~ |p_{2n+1}(x)-p_{2n+1,m}(x)| \le  2 \, \frac{\lVert F-F_m \rVert_\infty\, x^n}{(n-1)!}. 
\end{align*}

\end{lem}
\begin{proof} 
First we prove the assertion for $q_{2n}$. Since
\begin{align}
|\mu[r,x]-\mu_m[r,x]| \le 2 \, \lVert F-F_m \rVert_\infty ~~~ r,x \in [0,1], \label{l}
\end{align}
we get for $n=1$
\begin{align*}
|q_{2}(x) - q_{2,m}(x) | &=  \bigg|\int_0^x \int_0^t  \, dr \, d \mu (t) - \int_0^x \int_0^t \, dr \, d \mu_m (t) \bigg| \\
&= \bigg|\int_0^x  \mu [r,x]   -   \mu_m[r,x] \, dr  \bigg|  \\
&\le \int_0^x  |\mu [r,x]  -   \mu_m[r,x]|\,  dr  \\
&\le 2 \, \lVert F-F_m \rVert_\infty x,
\end{align*}
Thereby the assertion holds for $n=1$. Assume the assertion holds for $n \in \mathbb{N}$. Then
\begin{align*}
|q_{2n+2}(x) - q_{2n+2,m}(x)| &= \bigg|\int_0^x \int_0^t q_{2n}(r)  \, dr \, d \mu (t) - \int_0^x \int_0^t q_{2n,m}(r) \, dr \, d \mu_m (t) \bigg| \\
&= \bigg|\int_0^x  q_{2n}(r) \, \mu [r,x] \, dr  - \int_0^x  q_{2n,m}(r) \, \mu_m[r,x] \, dr  \bigg| \\
&\le \bigg|\int_0^x  (q_{2n}(r) - q_{2n,m}(r)) \, \mu [r,x] \, dr \bigg| \\ & ~~~+ \bigg|\int_0^x  q_{2n,m}(r) \, (\mu [r,x] - \mu_m[r,x] ) \, dr \bigg| \\
&\le  \int_0^x  |(q_{2n}(r) - q_{2n,m}(r))| \, dr + \int_0^x  q_{2n,m}(r) \, \big|\mu [r,x] - \mu_m[r,x] \big|\, dr. 
\end{align*}  
Because \[q_{2,m}(r) = \int_0^r \int_0^t \, dy\, d\mu_m (t) = \int_0^r \mu_m[y,r]\, dy \le r\]
and Lemma \ref{absch}, it follows \[q_{2n,m}(r) \le \frac{1}{n!} \, (q_{2,m}(r))^n \le \frac{r^n}{n!}.\]
Together with the induction hypothesis and \eqref{l} we get
\begin{align*}
|q_{2n+2}(x) - q_{2n+2,m}(x)| &\le  \lVert F-F_m \rVert_\infty \, \frac{2}{(n-1)!}  \int_0^x r^n \, dr + 2\, \lVert F-F_m \rVert_\infty\, \frac{1}{n!}  \int_0^x r^n\, dr \\
&=\lVert F-F_m \rVert_\infty \, \frac{2\, (n+1) \,  x^{n+1}}{(n+1)!} \\
&= 2 \lVert F-F_m \rVert_\infty \, \frac{ x^{n+1}}{n!}.
\end{align*}
For $p_{2n+1}$ the induction is the same as for $q_{2n}$. Therefore the assertion holds for $p_{2n+1}$. Then for $p_{2n}$, $n \ge 2$ we get
\begin{align*}
|p_{2n}(x)-p_{2n,m}(x)| &\le \int_0^x|p_{2n-1}(t) - p_{2n-1,m}(t)|\, dt  \\
&\le 2 \, \frac{\lVert F-F_m \rVert_\infty}{(n-2)!} \, \int_0^x t^{n-1}\, dt \\ 
&\le 2 \, \frac{\lVert F-F_m \rVert_\infty\, x^n}{(n-1)!}.
\end{align*}
We get the assertion for $q_{2n+1}$ analogously. The proof for $n=1$ is similar to the proof of the induction basis of $q_{2n}$.
\end{proof}

\begin{prop}\label{conv trig}
For all $z \in \mathbb{R}$ holds
\begin{align*}
\parallel \cqz - \cqzm\parallel_\infty &\le   c(z) \, \lVert F-F_m \rVert_\infty, ~~~~~~~~~~~~~~~~~ \parallel\cpz - \cpzm\parallel_\infty \le  c(z) \, \lVert F-F_m \rVert_\infty, \\
\parallel\sqz - \sqzm\parallel_\infty &\le  c(z)\,\lVert F-F_m \rVert_\infty, ~~~~~~~~~~\,~~~~~~~  \parallel\spz - \spzm\parallel_\infty \le c(z)\,\lVert F-F_m \rVert_\infty,
\end{align*}
where $c(z) > 0$ only depends on $z$.
\end{prop}
\begin{proof}
We show the assertion for $\cqz$ by applying Lemma \ref{pq absch}. Analogously we get the other assertions. For all $x \in [0,1]$ holds
\begin{align*}
|\cqz(x) - \cqzm(x)| &\le  \sum_{n=1}^\infty |q_{2n}(x)-q_{2n,m}(x)|\, z^{2n} \\
&\le  2 \,  \sum_{n=1}^\infty  \frac{\lVert F-F_m \rVert_\infty\, x^n}{(n-1)!}\, z^{2n} \\
&\le 2 \, z^2\, e^{z^2} \, \lVert F-F_m \rVert_\infty.
\end{align*}
\end{proof}
\begin{bemerkung}
Especially we get
\begin{align*}
|\cosq(z) - \cosqm(z)|&\le  c(z) \, \lVert F-F_m \rVert_\infty, ~~~~~ |\cosp(z) - \cospm(z)|\le   c(z) \, \lVert F-F_m \rVert_\infty, \\
|\sinq(z) - \sinqm(z)|&\le  c(z) \, \lVert F-F_m \rVert_\infty, ~~~~~~ |\sinp(z) - \sinpm(z)|\le   c(z) \, \lVert F-F_m \rVert_\infty. \\
\end{align*}
\end{bemerkung}
\begin{prop} \label{absch abl}
For all $z \in \mathbb{R}$ and all $m \in \mathbb{N}$ holds
\begin{align*}
|\sinq '(z)- \sinq_m ' (z)| &\le 2\, \lVert F-F_m \rVert_\infty \sum_{n=1}^\infty \frac{(2n+1)}{(n-1)!} \, z^{2n}, \\ ~~~ |\sinp '(z)- \sinp_m ' (z)| &\le 2 \lVert F-F_m \rVert_\infty \sum_{n=1}^\infty \frac{(2n+1)}{(n-1)!} \, z^{2n}.
\end{align*}
\end{prop}
\begin{proof}
The estimates are consequences of Lemma \ref{pq absch}.
\end{proof}
\begin{prop} \label{hilfsprop thm}
$(\sinp_n)_{n \in \mathbb{N}_0}$, $(\sinq_n)_{n \in \mathbb{N}_0}$, $(\cosp_n)_{n \in \mathbb{N}_0}$, $(\cosq_n)_{n \in \mathbb{N}_0}$ and $(\sinp_n')_{n \in \mathbb{N}_0}$, $(\sinq_n')_{n \in \mathbb{N}_0}$, $(\cosp_n')_{n \in \mathbb{N}_0}$, $(\cosq_n')_{n \in \mathbb{N}_0}$ converge uniformly on bounded intervals to $\sinp$, $\sinq$, $\cosp$, $\cosq$ and $\sinp'$, $\sinq'$, $\cosp'$, $\cosq'$, respectively.
\end{prop}
\begin{proof}
The statement follows from Proposition \ref{conv trig}, its proof and Proposition \ref{absch abl}, whereby an analogous statement to Proposition \ref{absch abl} holds for $\cosp'$ and $\cosq'$. 
\end{proof}
\begin{lem} \label{hilfslem thm}
Let $f: \mathbb{R} \rightarrow \mathbb{R}$ be continuously differentiable and $(f_n)_{n \in \mathbb{N}}$ be a sequence of continuously differentiable functions on $\mathbb{R}$ s.t. $f_n \rightarrow f$ and $f_n' \rightarrow f'$ 
uniformly on bounded intervals. If $f$ has exactly one zero $x \in (a,b)$ in $[a,b]$, $- \infty < a < b < \infty$ and if $f' \neq 0$ on $[a,b]$, then $f_n$ has exactly one zero in $[a,b]$ for all $n \ge n_0$.
\end{lem}
\begin{proof}
Let $x \in (a,b)$ be the unique zero of $f$ in $[a,b]$. Because $f_n \rightarrow f$ uniformly on $[a,b]$ and by assumption $f'(x) \neq 0$ we have at least one zero $x_n \in (a,b)$ of $f_n$ for each $n \ge n_0$. Therefore it is sufficient to show that this zero is unique in $[a,b]$. Suppose, there are infinite many $n \in \mathbb{N}$ s.t. $f_n$ has at least two zeros in $[a,b]$, i.e. there exists a subsequence $(n_k)_{k \in \mathbb{N}}$ s.t. $x_{n_k},\, y_{n_k} \in [a,b]$ with $x_{n_k} \neq y_{n_k}$ and $f_{n_k}(x_{n_k}) = f_{n_k}(y_{n_k})=0$ for all $k \in \mathbb{N}$. W.l.o.g. let $x_{n_k} < y_{n_k}$ for all $k$. Because
\begin{align*}
|f(x_{n_k})| = |f(x_{n_k}) - f_{n_k} (x_{n_k})| \le \sup_{y \in [a,b]} |f(y) - f_{n_k}(y) | \stackrel{k \rightarrow \infty}{\longrightarrow} 0
\end{align*}
and the Taylor formula (together with the mean value theorem)
\begin{align*}
|f(x_{n_k})| &= |f(x) + f'(\zeta_k) \, (x_{n_k}- x)|\\& = |f'(\zeta_k) \, (x_{n_k}-x)| \\&\ge \min_{x \in [a,b]} |f'(x)| |  x-x_{n_k}|, ~~~~~~~~~ \zeta_k \in (x_{n_k},x)
\end{align*}
and $f'\big|_{[a,b]} \neq 0$, we get
\begin{align*}
|x_{n_k} - x| \le \frac{|f(x_{n_k})|}{\min\limits_{x \in [a,b]} |f'(x)|} \stackrel{k \rightarrow \infty}{\longrightarrow} 0.
\end{align*}
Analogously we get $\lvert~\cdot~ \rvert-\lim\limits_{k \rightarrow \infty} y_{n_k} = x$. Moreover, Taylor's formula implies that there exists $\xi_k \in (x_{n_k}, y_{n_k})$  s.t.
\begin{align*}
0 &= f_{n_k} (x_{n_k}) \\
&= f_{n_k}(y_{n_k}) + f_{n_k}'(\xi_k) \, (x_{n_k} - y_{n_k}) \\
&= f_{n_k}'(\xi_k) \, (x_{n_k} - y_{n_k})
\end{align*}
and therefore, because $x_{n_k} \neq y_{n_k}$, we get $f_{n_k}'(\xi_k) = 0$ for all $k \in \mathbb{N}$. Let $\epsilon >0$. Then, because $\xi_k \rightarrow x$, follows
\begin{align*}
|f'(x)| &= |f'(x) - f'_{n_k}(\xi_{k})| \\
&\le |f'(x) - f'(\xi_{k})| + |f'(\xi_k) - f_{n_k}'(\xi_k)| \\
&\le   |f'(x) - f'(\xi_{k})| + \sup_{y \in [a,b]}|f'(y) - f_{n_k}'(y)| \\
&\le \epsilon
\end{align*}
for $k \ge k_0$. Thereby the last estimate follows, because $f'$ is continuous and $(f'_n)_{n \in \mathbb{N}}$ converge to $f'$ uniformly on bounded intervals. Because $\epsilon > 0$ is arbitrary, it follows that $f'(x) = 0$. This is a contradiction to the assumption.
\end{proof}

\begin{thm} \label{zentral}
For all $m \in \mathbb{N}_0$ holds
\begin{align*}
|\lambda_{N,m} - \lambda_{N,m,n}|\le  c(m) \lVert F-F_n \rVert_\infty ~~~ \text{for all } n \ge n_0(m)
\end{align*}
and for all $m \in \mathbb{N}$ holds
\begin{align*}
|\lambda_{D,m}- \lambda_{D,m,n}|\le c(m)\lVert F-F_n \rVert_\infty~~~ \text{for all } n \ge n_0(m),
\end{align*}
where $c(m) > 0$ and $n_0(m) \in \mathbb{N}$ only depend on $m$.
\end{thm}
\begin{proof}
We show the statement for $\lambda_{N,m}$, $m \in \mathbb{N}_0$. The proof for $\lambda_{D,m}$ works analogously. For $m=0$ the statement is obvious. Let $m \in \mathbb{N}$ and $z_m \coloneqq \sqrt{\lambda_{N,m}}$.
Applying Proposition \ref{ungleichnull}, we have $\sinp '(z_m) \neq 0$. Because the zeros of $\sinp$ are countable and have no finite accumulation points we get  
\begin{align}
0 \neq \sgn(\sinp(z_m-\epsilon))  \neq \sgn(\sinp(z_m+\epsilon)) \neq 0 \label{vorz}
\end{align}
for all $0 < \epsilon < \hat{\epsilon}$, $\hat{\epsilon}>0$ sufficiently small. 
Because $\sinp'$ is continuous and  $\sinp'(z_m) \neq 0$, there exists a sufficiently small $\epsilon$ neighbourhood of $z_m$ s.t. $\sinp'|_{[z_m-\epsilon,z_m + \epsilon]} \neq 0$. Thereby let $\epsilon$ be s.t. \eqref{vorz} holds and s.t. $z_m$ is the unique zero of $\sinp$ on $[z_m-\epsilon,z_m + \epsilon]$. 

Then Proposition \ref{hilfsprop thm} implies, that there exists a unique $z_{m,n} \in [z_m-\epsilon,z_m + \epsilon]$ for all $n \ge \hat{n}$ s.t. $\sinp_n(z_{m,n})=0$. Applying Taylor's formula, there exists a $\theta_n \in (z_{m,n},z_n)$ with
\[|\sinp(z_{m,n})| = |\sinp'(\theta_n) (z_{m,n}-z_n)| \ge \min_{y \in [z_m-\epsilon,z_m + \epsilon]} |\sinp'(y)|\, |z_{m,n}-z_m| \]
and therefore $z_{m,n} \longrightarrow z_m$ for $n \rightarrow \infty$, whereby $\sinp(z_{m,n}) \longrightarrow 0$ because 
\begin{align*}
|\sinp(z_{m,n})| &= |\sinp(z_{m,n}) - \sinp_n(z_{m,n})| \\
&\le \sup_{y \in [z_m-\epsilon,z_m + \epsilon]}|\sinp(y) - \sinp_n(y)| \stackrel{n \rightarrow \infty}{\longrightarrow} 0 
\end{align*}
holds. Let $\delta>0$ and $n \ge n_0$, $n_0$ sufficiently large, s.t.
\begin{align}
|z_{m}- z_{m,n}| < \delta. \label{delta}
\end{align}
To complete the proof, we first have to show inductively, that $z_{m,n}^2 = \lambda_{N,m,n}$ for all $n \ge \tilde{n}(m)$. Thereby the assertion is obvious for $m=0$. Assume, that the assertion holds for $m \in \mathbb{N}_0$.
Because of $\sinp'|_{[z_m-\epsilon,z_m + \epsilon]} \neq 0$, $\sinp \big|_{[z_m-\epsilon, z_m + \epsilon]\backslash \{z_m\}}\neq 0$ and Proposition \ref{hilfsprop thm}, we can apply Lemma \ref{hilfslem thm}. This implies, that just a finite number of $\sinp_n$ have more than one zero in $[z_m-\epsilon,z_m + \epsilon]$. Analogously we get $\sinp'|_{[z_{m+1}-\tilde{\epsilon},z_{m+1} + \tilde{\epsilon}]} \neq 0$ for a sufficiently small $\tilde{\epsilon} > 0$. This implies that also just a finite number of $\sinp_n$ have more than one zero in $(z_{m+1}-\tilde{\epsilon},z_{m+1} + \tilde{\epsilon})$. Applying the uniformly convergence on $[z_m+\epsilon, z_{m+1} - \tilde{\epsilon}]$ and $\sinp|_{[z_m+\epsilon, z_{m+1} - \tilde{\epsilon}]} \neq 0$, we get that just a finite number of $\sinp_n$ have a zero in $[z_m+\epsilon, z_{m+1} - \tilde{\epsilon}]$. Let $\tilde{n}(m+1)$ be minimal s.t. $z_{m,n}^2 = \lambda_{N,m,n}$, $\sinp_n$ has a unique zero in $[z_m-\epsilon,z_m + \epsilon]$, $\sinp_n$ has no zero in $[z_m+\epsilon, z_{m+1} - \tilde{\epsilon}]$ and $\sinp_n$ has a unique zero in $[z_{m+1}-\tilde{\epsilon},z_{m+1} + \tilde{\epsilon}]$ for all $n \ge \tilde{n}(m+1)$. Thereby the assertion follows.
\\\\
Moreover, let $n_0$ be s.t. $z_{m,n}^2 = \lambda_{N,m,n}$ for all $n \ge n_0$. With Taylor's formula there exists a $\xi_n$ between $z_{m}$ and $z_{m,n}$ s.t. 
\begin{align}
\sinp_n(z_m) &= \sinp_n (z_{m,n}) + \sinp'_n(\xi_n) \, ( z_m -z_{m,n}) \notag \\
&= \sinp'_n(\xi_n) \, (z_m-  z_{m,n}). \label{taylor}
\end{align} 
Let $\epsilon_1 > 0$, $\delta > 0$ sufficiently small s.t. \eqref{delta} implies
\begin{align}
|\sinp'(z_m) - \sinp'(\xi_n)| < \epsilon_1
\end{align}
for $n \ge n_0$. This is possible, because $\sinp'$ is continuous and
\begin{align*}
|z_m - \xi_n| \le |z_m - z_{m,n}| < \delta
\end{align*}
for $n \ge n_0$ holds. Thereby we get for $n \ge n_1$, $n_1 \in \mathbb{N}$ sufficiently large, \[| \sinp_n' (\xi_n)| \ge \frac{|\sinp'(z_m)|}{2} > 0.\] 
Together with \eqref{taylor} and Proposition \ref{conv trig}, we have 
\begin{align*}
\frac{|\sinp'(z_m)|}{2} \, |z_m - z_{m,n}| &\le | \sinp_n' (\xi_n)|\, |z_m - z_{m,n}| \\
&= |\sinp_n(z_m)| \\
&= |\sinp_n(z_m) - \sinp (z_m)| \\
&\le \lVert F-F_n \rVert_\infty \, \tilde{c}(z_m).
\end{align*}
\end{proof}
In the following, we denote the $m$-th Neumann and Dirichlet eigenfunction of $\frac{d}{d\mu_n}\frac{d}{d x}$ by $f_{N,m,n}$ and $f_{D,m,n}$, respectively.
\begin{thm} \label{zentral eigenfct}
For all $m \in \mathbb{N}_0$ holds
\begin{align*}
\lVert f_{N,m}-f_{N,m,n}\rVert_\infty \le  c(m) \lVert F-F_n \rVert_\infty ~~~ \text{for all } n \ge n_0(m)
\end{align*}
and for all $m \in \mathbb{N}$ holds
\begin{align*}
\lVert f_{N,m}-f_{N,m,n}\rVert_\infty\le c(m)\lVert F-F_n \rVert_\infty~~~ \text{for all } n \ge n_0(m),
\end{align*}
where $c(m) > 0$ and $n_0(m) \in \mathbb{N}$ only depend on $m$.
\end{thm}
\begin{proof}
It holds
\begin{align*}
|f_{N,m,n}(x) - \text{cp}_{\lambda^{1/2}_{N,m},n}(x)| \le \sum_{k=0}^\infty |\lambda_{N,m,n}^k - \lambda_{N,m}^k| p_{2k,n}(x).
\end{align*}
For $z, z' \in \mathbb{R}$ we have with a generalized binomial formula 
\begin{align*}
z^{k+1} - (z')^{k+1} = (z-z') \sum_{j=0}^k z^j (z')^{k-j}
\end{align*}
and thus if $|z - z'| \le 1$, we get
\begin{align*}
\left| z^{k+1} - (z')^{k+1} \right| \le (k+1)\left|z-z'\right|(z+1)^k.
\end{align*}
Since $\lambda_{N,m,n} \rightarrow \lambda_{N,m}$ by Theorem \ref{zentral}, we can choose $n_0 = n_0(m)$ large enough such that $|\lambda_{N,m,n} - \lambda_{N,m}|\le 1$ for $n \ge n_0$ and thus 
\begin{align*}
|f_{N,m,n}(x) - \text{cp}_{\lambda^{1/2}_{N,m},n}(x)| \le |\lambda_{N,m,n} - \lambda_{N,m}| \sum_{k=0}^\infty (k+1) (\lambda_{N,m} + 1)^k p_{2k,n}(x).
\end{align*}
By Lemma \ref{absch} the last sum is convergent and thus we can conclude the claim.
\end{proof}
\section{Eigenvalue approximation for Cantor Measures}
In this section we use the results of the previous section to give the speed of convergence of the eigenvalues and -functions of approximations of Cantor measures. Therefore, let $\mu^w$ be the unique invariant Borel probability measure on the unit interval induced by the IFS $\s = (S_1,S_2)$, $S_1(x) = \frac{1}{3}x$, $S_2(x) = \frac{1}{3}x + \frac{2}{3}$, $x \in [0,1]$ and weight vector $w = (w_1,w_2)$, $w_1 \in (0,1)$, $w_2 = 1-w_1.$ For reasons of simplicity we only consider the classical Cantor set, but the following concept can be modified to Cantor like sets.
W.l.o.g. let $w_1 \le w_2$. For $n \in \mathbb{N}$ we define $\mu_n^w: ([0,1],B[0,1]) \longrightarrow [0,1]$ by
\begin{align}
\mu_n^w(A) \coloneqq 3^n \! \! \! \sum_{x \in \{1,2\}^n} \lambda^1_{|_{I_x}}(A) \, \prod_{i=1}^n w_{x_i}, ~~~~ A \in B[0,1], \label{mass}
\end{align}
where $I_x \coloneqq (S_{x_1} \circ ... \circ S_{x_n})([0,1])$, $x \in \{1,2\}^n$ and $B[0,1]$ denotes the Borel $\sigma$-Algebra on $[0,1]$. Figure 1 shows how $\mu_n^w$ weights the intervals in the $n$-th approximation step $E_n$ of the Cantor set $F$. Remark, that the attractor of the given IFS is F. This implies $\supp \mu^w = F$. Then $\mu_n^w$, $n \in \mathbb{N}$ is a non-atomic Borel probability measure and the identity
\[\mu_n^w [0,y] = w_1 \, \mu_{n-1}^w [0,3y] + w_2 \, \mu_{n-1}^w [0,3y-2], ~~~~ y \in [0,1], \]
where $\mu_n^w[b,a] \coloneqq -\mu_n^w[a,b]$, $a < b$, $\mu_0^w \coloneqq \lambda^1_{|_{[0,1]}}$ holds. Furthermore, it is well known that $(\mu_n)^w$ converges weakly to $\mu^w$.
\begin{figure}[h]
\centering
\includegraphics[width=1.0\textwidth]{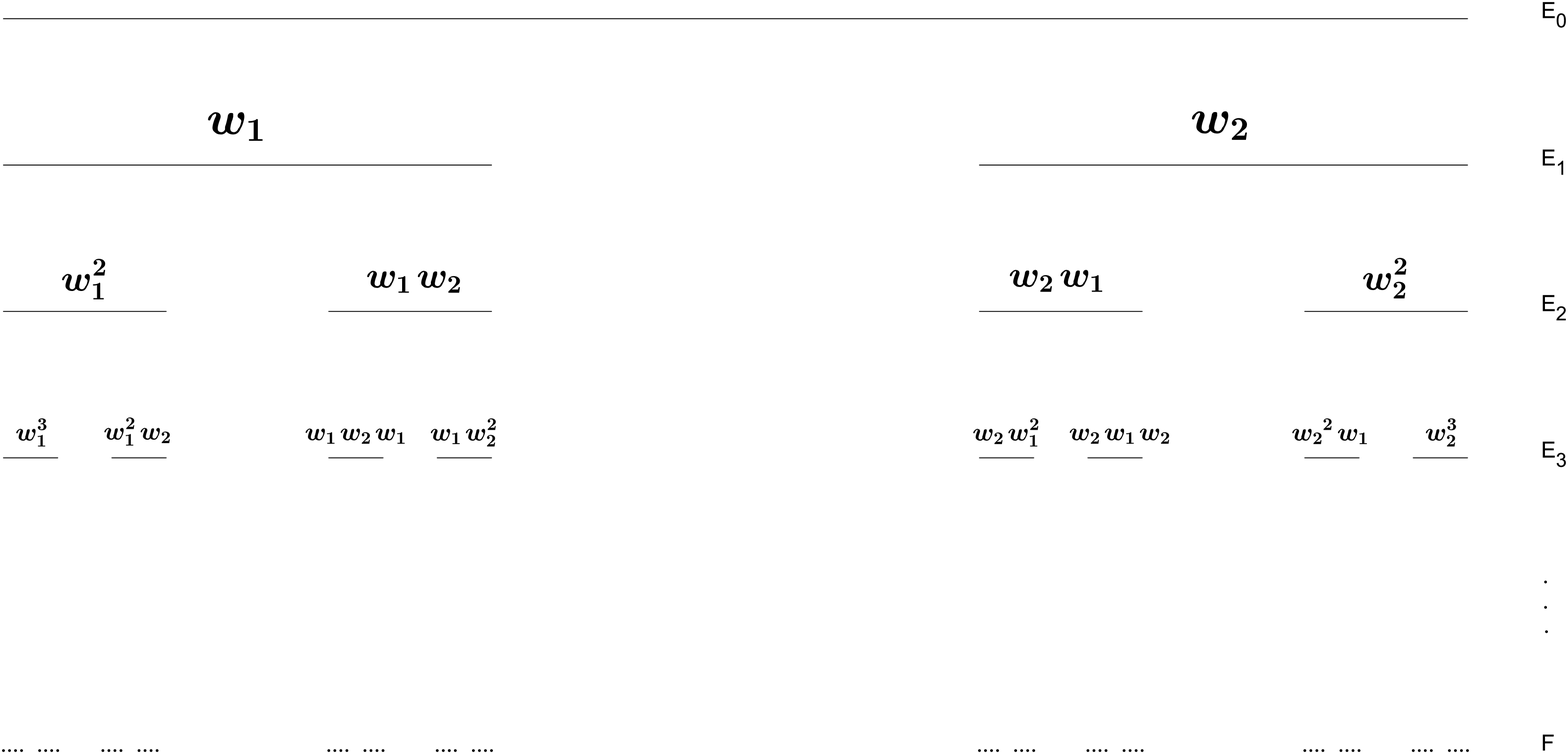}
\caption{weighted Cantor set}
	\label{fig1}\
\centering	\
\end{figure}

\begin{lem} \label{konv}
There exists a Borel probability measure $\mu$ on $[0,1]$ such that \[\parallel F_\mu - F_{\mu^w_n}\parallel_{\infty}\le  \frac{w_2^n}{w_1}, ~~~ \text{for all } n \in \mathbb{N},\]
where $F_\nu$ is the distribution function for given Borel measure $\nu$.
\end{lem}
\begin{proof}
First we show
\begin{align}
|F_{\mu_n^w}(t) - F_{\mu_{n+1}^w}(t)| \le w_2^n, ~~~~ t \in [0,1]. \label{abschae}
\end{align}
Therefore let $x \in \{1,2\}^n$, $n \in \mathbb{N}$ and $y \in \partial I_x$. We have by definition
\begin{align}
F_{\mu_n^w}(y) = F_{\mu_{n+1}^w}(y). \label{gleich}
\end{align}

Also $F_{\mu_n^w}$ and $F_{\mu_{n+1}^w}$ are constant and equal on $[0,1]\backslash E_n$. Therefore it is sufficient to show the statement on $E_n$. Let $y \in \partial I_x$ be the left boundary of $I_x$. Because of \eqref{gleich} we get
\begin{align*}
|F_{\mu_{n}^w}(t)-F_{\mu_{n+1}^w}(t)| &= |F_{\mu_{n}^w}(t)-F_{\mu_{n+1}^w}(t)-F_{\mu_{n}^w}(y)+F_{\mu_{n+1}^w}(y)| \\
&= |\mu_n^w[y,t]- \mu_{n+1}^w[y,t]|. 
\end{align*}  
Let $t \in I_{(x_1,...,x_n,1)}$. Then
\begin{align*}
|\mu_n^w[y,t]- \mu_{n+1}^w[y,t]| &= 3^n \, (t-y) \prod_{i=1}^n w_{x_i} \, |\, 1-3w_1\,|  \\ &\le \prod_{i=1}^n w_{x_i} \\ &\le w_2^n.
\end{align*}
If $t \in I_x \backslash (I_{(x_1,...,x_n,1)} \cup I_{(x_1,..,x_n,2)})$, then
\begin{align*}
|\mu_n^w[y,t]- \mu_{n+1}^w[y,t]| &= \bigg|\, 3^n \, (t-y) \prod_{i=1}^n w_{x_i} - \prod_{i=1}^{n+1}w_{x_i}\, \bigg| \\
&= \prod_{i=1}^n w_{x_i} \, |\, 3^n \, (t-y) - w_1\, | \\
&\le \prod_{i=1}^n w_{x_i} \\
&\le w_2^n,
\end{align*}
whereby $(t-y) \le \frac{2}{3^{n+1}}$ has been used. 
If $t \in I_{(x_1,...,x_n,2)}$, then $(t-y) \le \frac{1}{3^n}$ and $(t-z) \le \frac{1}{3^{n+1}}$, where $z \in \partial I_{(x_1,...,x_n,2)}$ is the left boundary of $I_{(x_1,...,x_n,2)}$. Hence we get
\begin{align*}
|\mu_n^w[y,t]- \mu_{n+1}^w[y,t]|  &= \bigg|\, 3^n \, (t-y) \prod_{i=1}^n w_{x_i} - w_1 \prod_{i=1}^n w_{x_i} - 3^{n+1} \, (t-z) \, w_2 \prod_{i=1}^n w_{x_i} \, \bigg| \\
&= \prod_{i=1}^n w_{x_i} \, \big|\, 3^n \, (t-y)  - w_1 - 3^{n+1} \, (t-z) \, w_2 \, \big| \\
&\le \prod_{i=1}^n w_{x_i} \\
&\le w_2^n.
\end{align*}
Since $x \in \{1,2\}^n$ is arbitrary, the statement follows on $E_n$ and therefore \eqref{abschae}. Thus
\begin{align*}
\big|\big| F_{\mu_{n}^w}-F_{\mu_{n+1}^w} \big|\big|_\infty = \sup_{x \in [0,1]} \big| F_{\mu_{n}^w}(x)-F_{\mu_{n+1}^w}(x) \big| \le w_2^n.
\end{align*}
For $k \in \mathbb{N}$ we get iteratively
\begin{align*}
\big|\big| F_{\mu_{n}^w}-F_{\mu_{n+k}^w} \big|\big|_\infty &\le \big|\big| F_{\mu_{n}^w}-F_{\mu_{n+k-1}^w} \big|\big|_\infty + \big|\big| F_{\mu_{n+k}^w}-F_{\mu_{n+k-1}^w} \big|\big|_\infty \\
&\le  \big|\big| F_{\mu_{n}^w}-F_{\mu_{n+k-2}^w} \big|\big|_\infty + \big|\big| F_{\mu_{n+k-1}^w}-F_{\mu_{n+k-2}^w} \big|\big|_\infty + w_2^{n+k-1} \\
&\le \sum_{j=n}^{n+k-1} w_2^j \\
&\le \sum_{j=n}^{\infty} w_2^j \\ 
&= \frac{1}{w_1} - \frac{1-w_2^{n}}{w_1} = \frac{1}{w_1} \, w_2^{n}.
\end{align*}
Hence $(F_{\mu_n^w})_{n \in \mathbb{N}}$ is a Cauchy sequence on the Banach Space $(C^0([0,1]),\parallel\cdot\parallel_\infty)$. Thus the limit $F_\mu \coloneqq ~ \parallel\cdot \parallel_{\infty}- \lim\limits_{m \rightarrow \infty} F_{\mu_m^w}$ exists in $C^0([0,1])$. 
Especially $(\mu_n^w)_{n \in \mathbb{N}}$ converge weakly to a Borel probability measure on $[0,1]$. Furthermore follows
\begin{align*}
\big|\big| F_{\mu}-F_{\mu_{n}^w} \big|\big|_\infty &= \lim_{m \rightarrow \infty} \big|\big| F_{\mu_{m}^w}-F_{\mu_{n}^w} \big|\big|_\infty \\
&\le \lim_{m \rightarrow \infty} \frac{1}{w_1} \, w_2^{\min\{n,m\}} \\
&=  \frac{1}{w_1} \, w_2^{n}.
\end{align*}
Hence the claim follows.
\end{proof}
Since $\mu^w$ is the weak limit of $(\mu_n^w)_n$, we get with Lemma \ref{konv}

\begin{prop} \label{konv inv}
It holds \[\parallel F_{\mu^w} - F_{\mu^w_n}\parallel_{\infty} \le \frac{w_2^n}{w_1}.\]
\end{prop}
With Theorem \ref{zentral} and Theorem \ref{zentral eigenfct} we therefore get
\begin{thm}
For all $m \in \mathbb{N}_0$ holds
\begin{align*}
|\lambda_{N,m} - \lambda_{N,m,n}|\le  c(m)  w_2^n, ~~~~ \lVert f_{N,m}-f_{N,m,n}\rVert_\infty \le  c(m) w_2^n  ~~~ \text{for all } n \ge n_0(m)
\end{align*}
and for all $m \in \mathbb{N}$ holds
\begin{align*}
|\lambda_{D,m} - \lambda_{D,m,n}|\le  c(m) w_2^n, ~~~~ \lVert f_{N,m}-f_{N,m,n}\rVert_\infty\le c(m) w_2^n~~ \text{for all } n \ge n_0(m),
\end{align*}
where $c(m) > 0$ and $n_0(m) \in \mathbb{N}$ only depend on $m$.
\end{thm}

\pagebreak
The following figures show the approximation of the Neumann and Dirichlet eigenvalues and the approximation of the eigenfunctions for the special case $w = (1/2,1/2)$.

\begin{figure}[h]
\begin{center}
 {\includegraphics[width=0.35\textwidth]{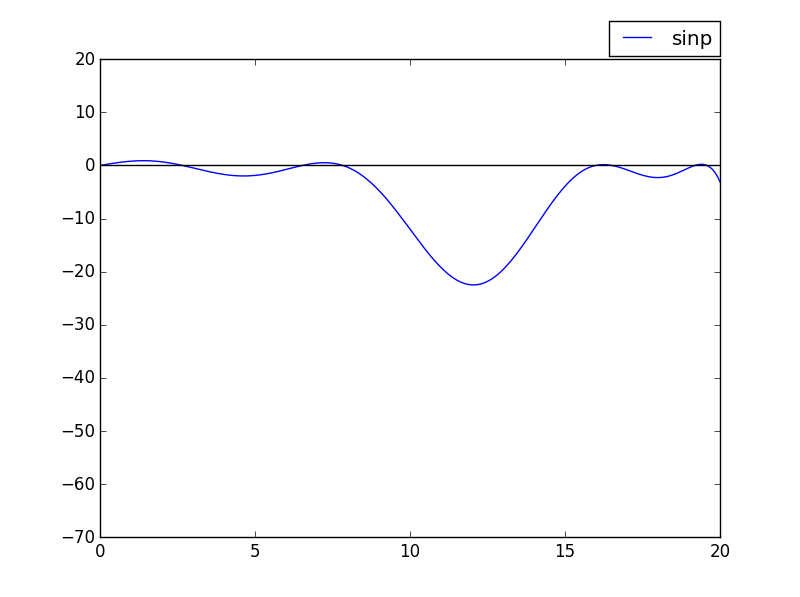}}
{\includegraphics[width=0.35\textwidth]{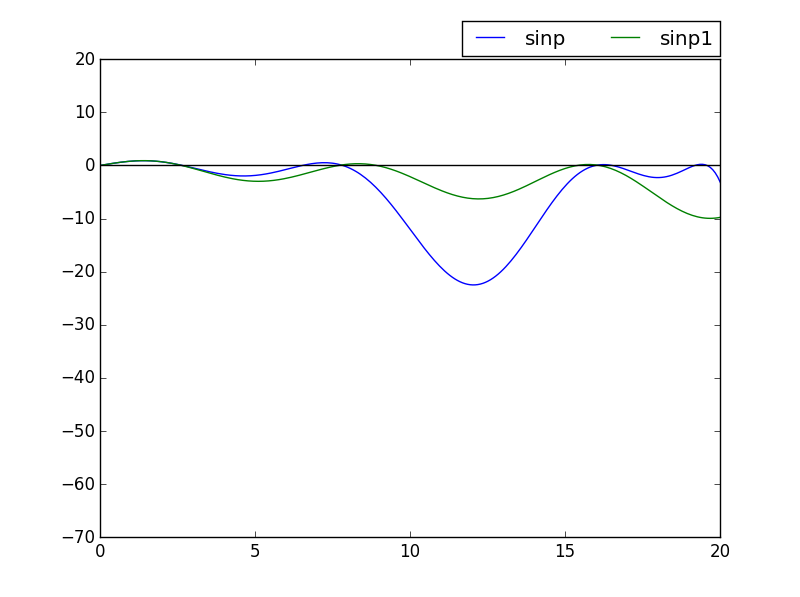}} 
{\includegraphics[width=0.35\textwidth]{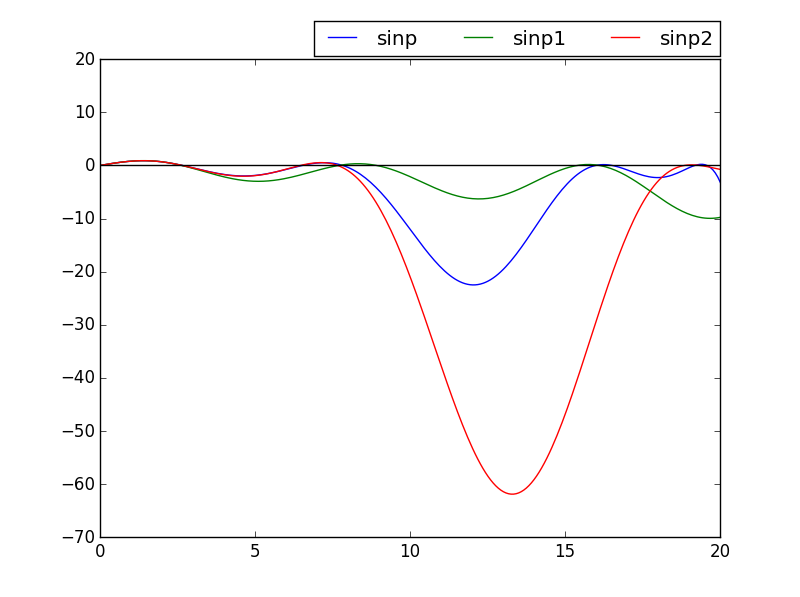}}
 {\includegraphics[width=0.35\textwidth]{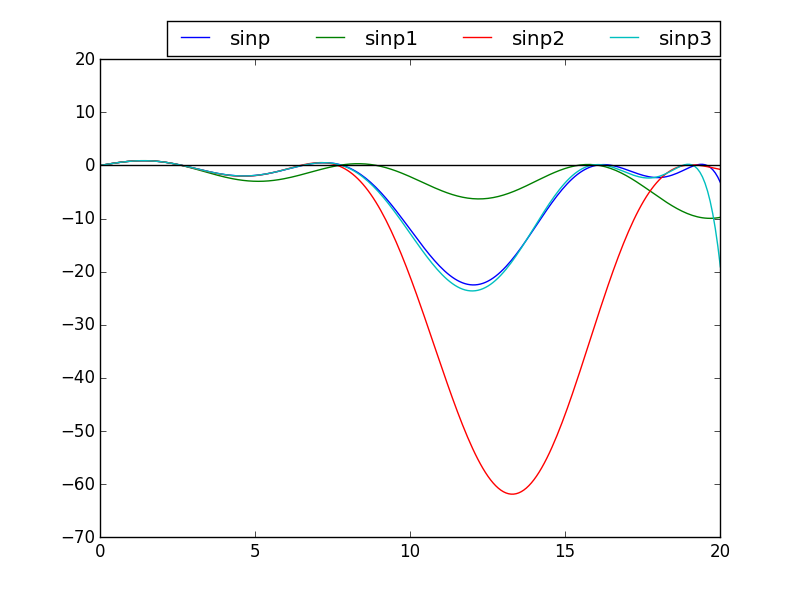}} 
\vspace{10pt}

{\includegraphics[width=0.35\textwidth]{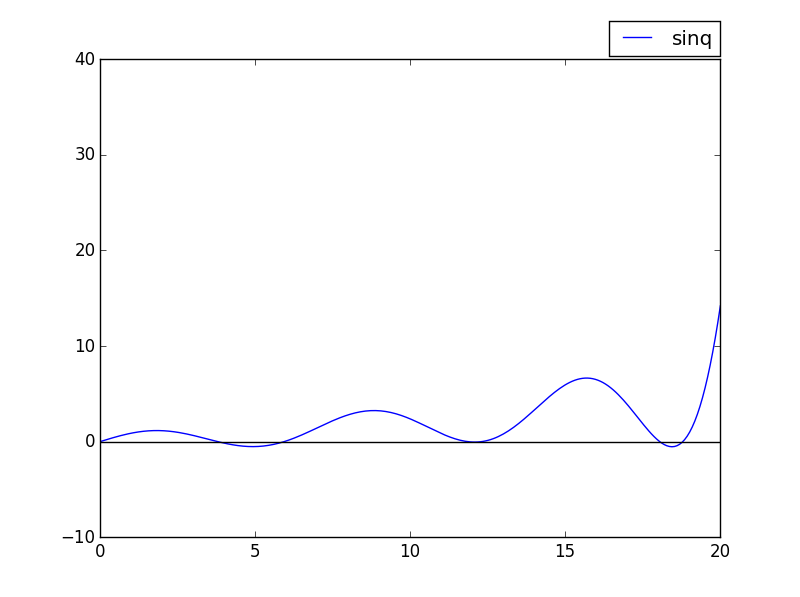}}
 {\includegraphics[width=0.35\textwidth]{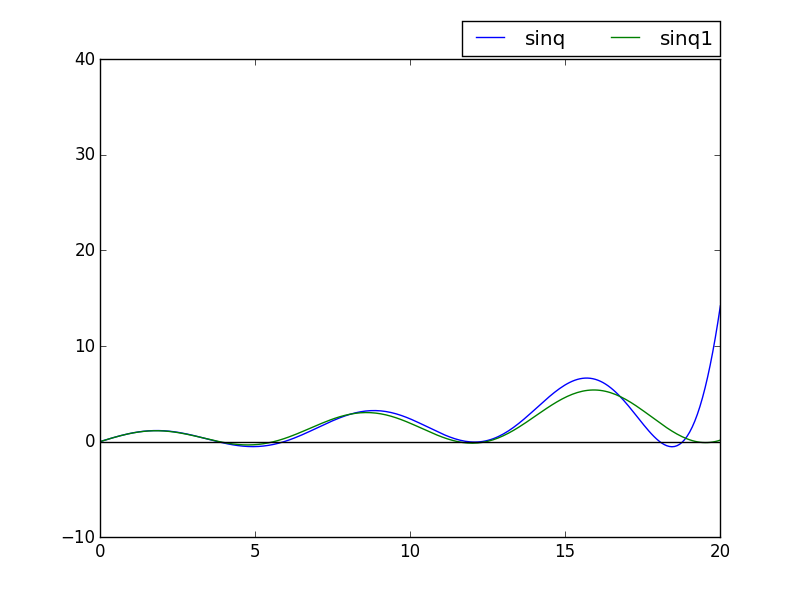}}
{\includegraphics[width=0.35\textwidth]{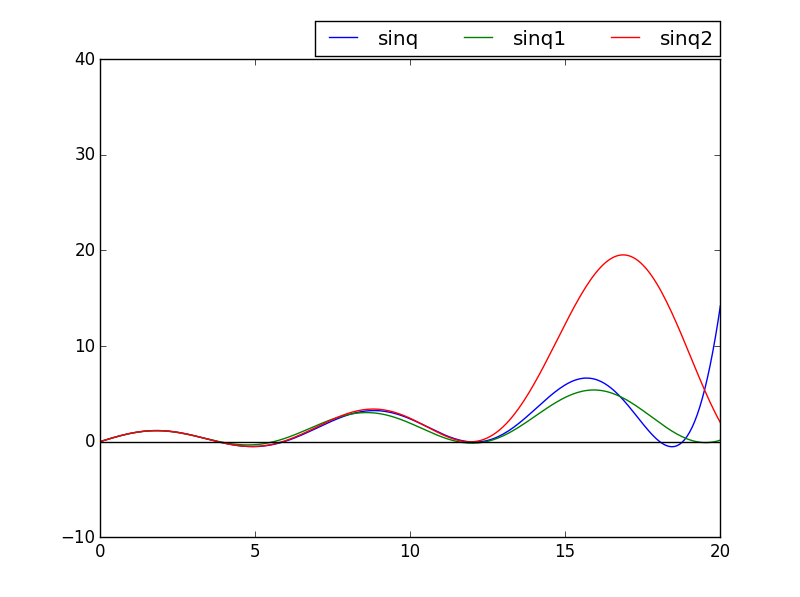}}
 {\includegraphics[width=0.35\textwidth]{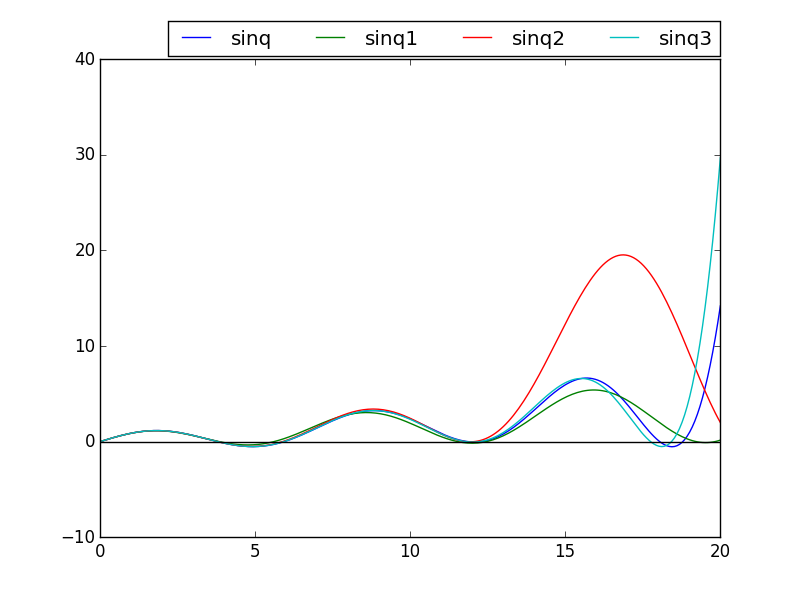}}
 \caption{Sinp and Sinq functions}
	\label{fig3}
	\end{center}
\end{figure}
$\sinp$i are the $\sinp$ functions w.r.t. $\mu_i^w$, $i=1,2,3$ and $\sinp$ is the $\sinp$ function w.r.t. $\mu^w$.
$\sinq$i are the $\sinq$ functions w.r.t. $\mu_i^w$, $i=1,2,3$ and $\sinq$ is the $\sinq$ function w.r.t. $\mu^w$. 
\newpage
The following figures show the first six Neumann and Dirichlet eigenfunctions. Thereby fN and fD are the Neumann and Dirichlet eigenfunctions w.r.t. $\mu^w$, respectively and fNi and fDi are the Neumann and Dirichlet eigenfunctions w.r.t. $\mu^w_i$, i=1,2, respectively. The $n_{\text{th}}$ Neumann and Dirichlet eigenfunction has exactly $n$ and $n+1$ zeros in $[0,1]$, $n=1,...,6$, respectively. 

\begin{figure}[h]
\begin{center}
 \captionsetup[subfigure]{labelformat=empty}
{\includegraphics[width=0.45\textwidth]{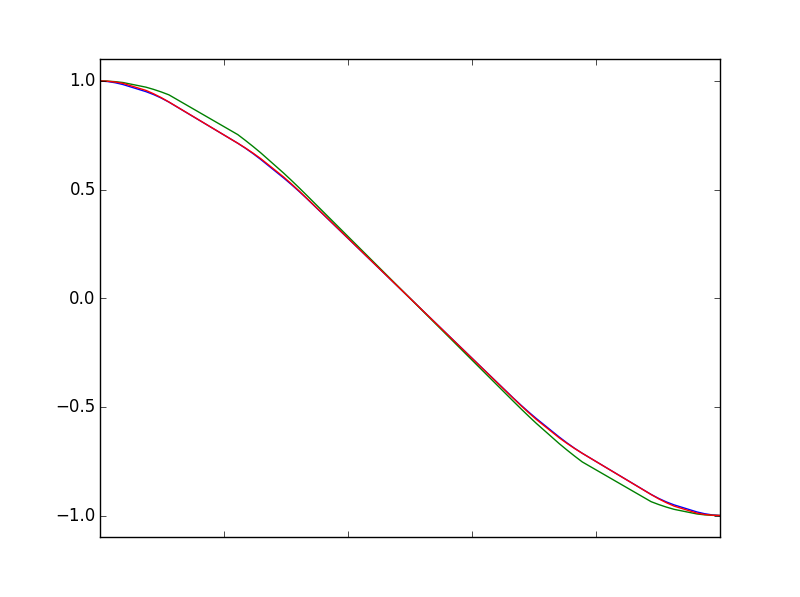}}
{\includegraphics[width=0.45\textwidth]{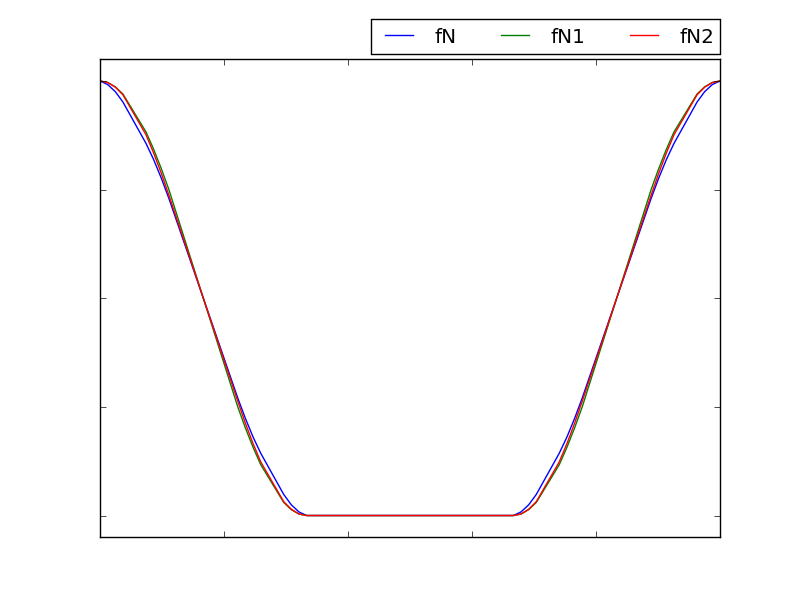}}  
{\includegraphics[width=0.45\textwidth]{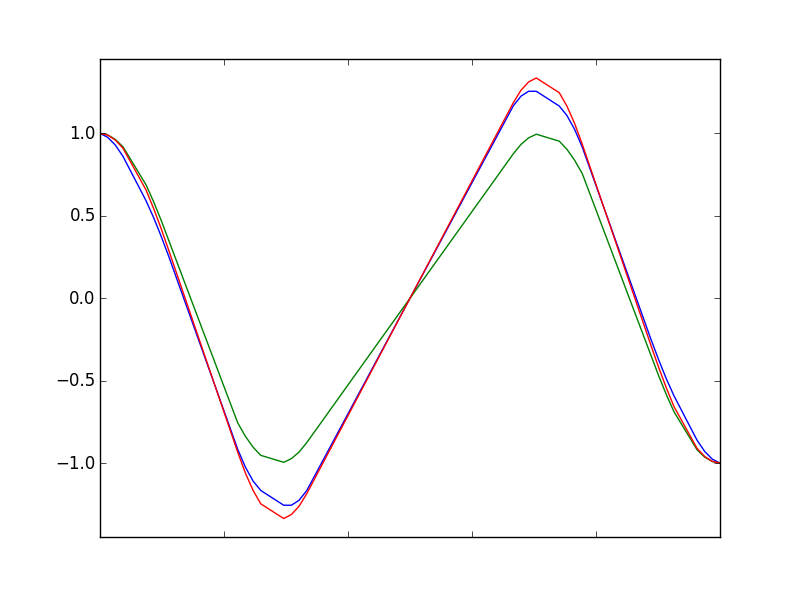}}
{\includegraphics[width=0.45\textwidth]{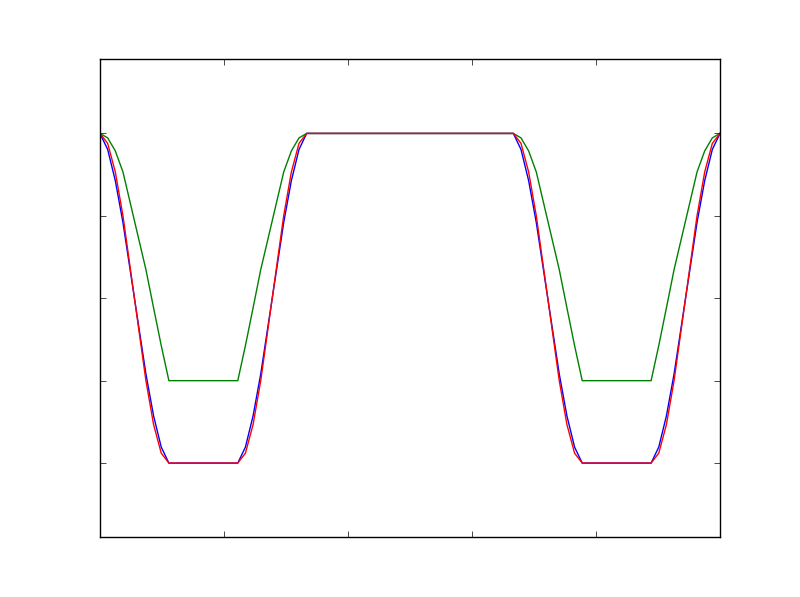}} 
{\includegraphics[width=0.45\textwidth]{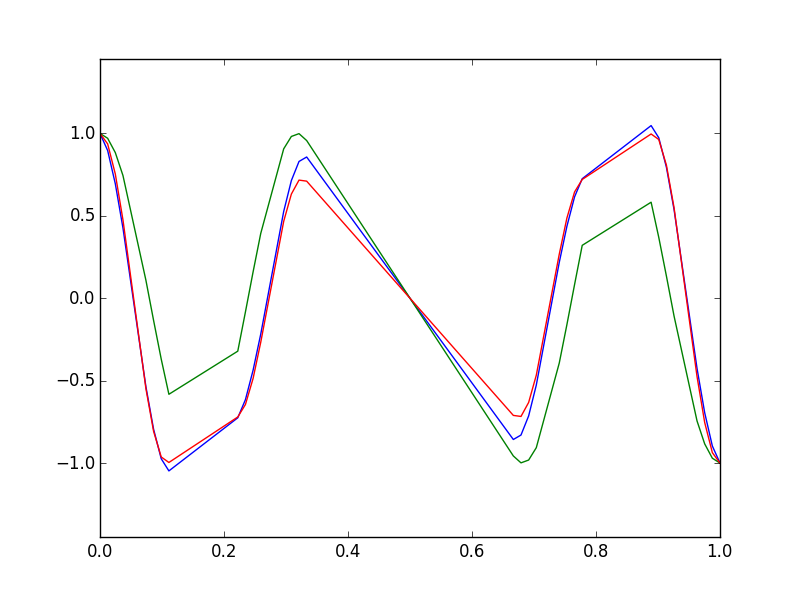}}
{\includegraphics[width=0.45\textwidth]{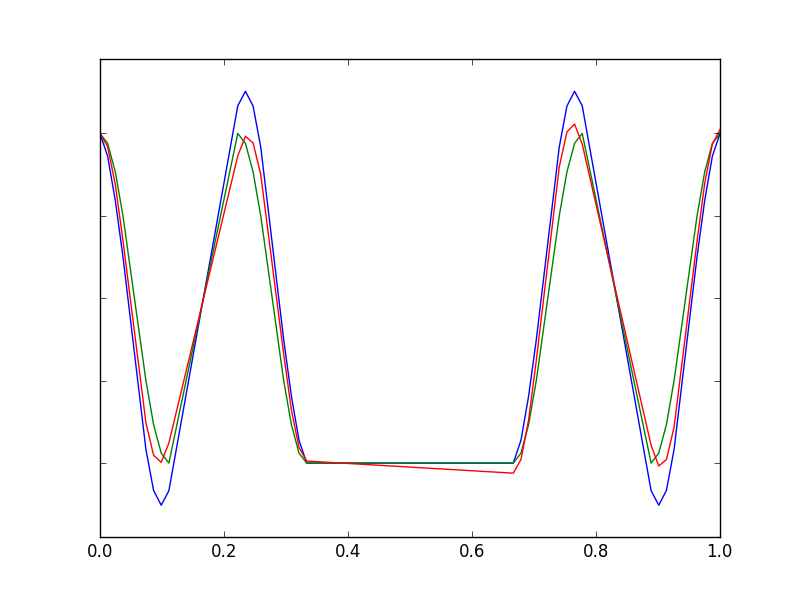}}
 \caption{Neumann eigenfunctions}
 \label{fig7}
 \end{center}
\end{figure}
\newpage
\begin{figure}[h]
 \captionsetup[subfigure]{labelformat=empty}
 \begin{center}
{\includegraphics[width=0.45\textwidth]{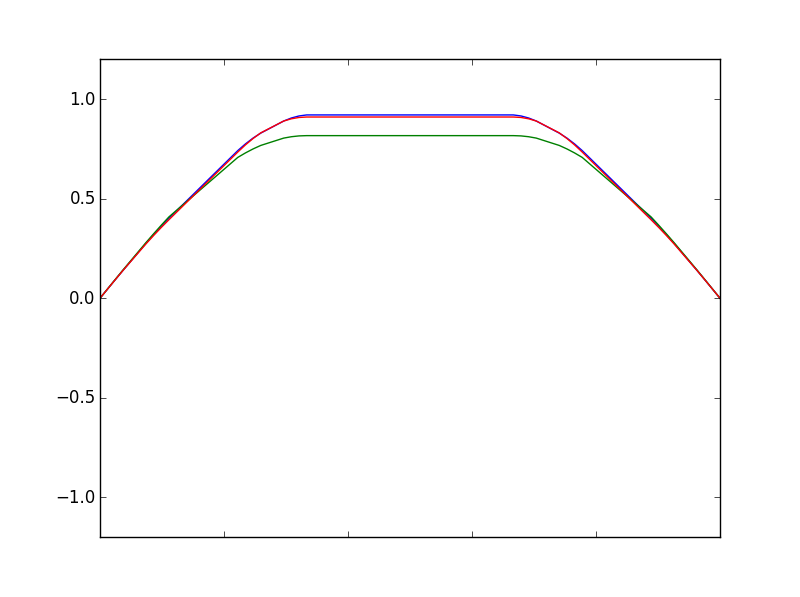}}
{\includegraphics[width=0.45\textwidth]{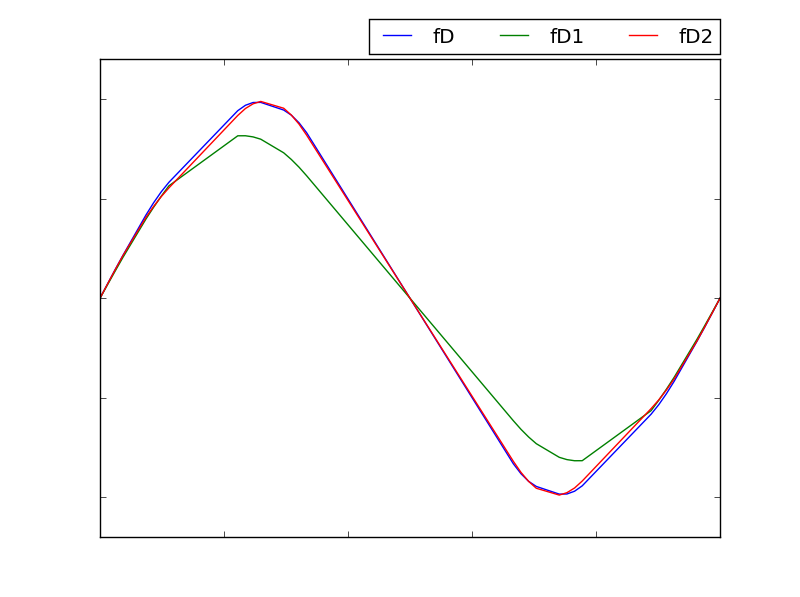}} 
{\includegraphics[width=0.45\textwidth]{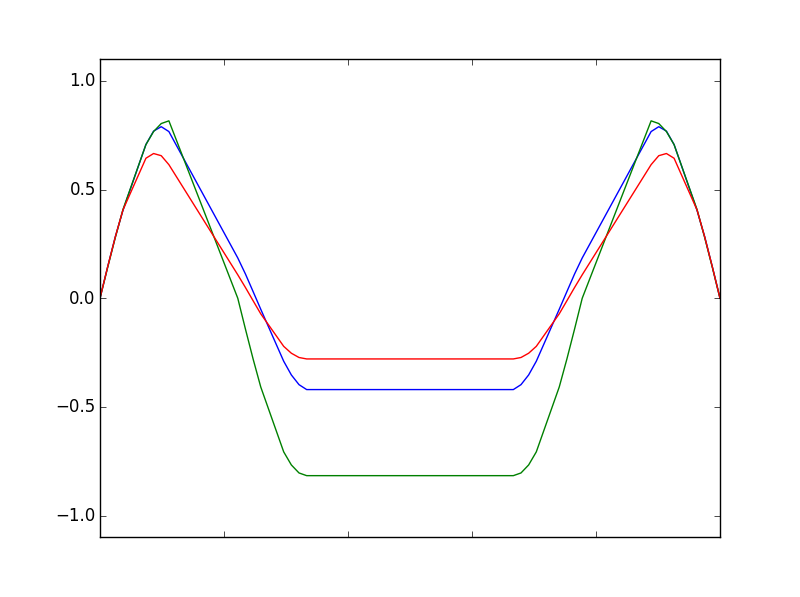}}
{\includegraphics[width=0.45\textwidth]{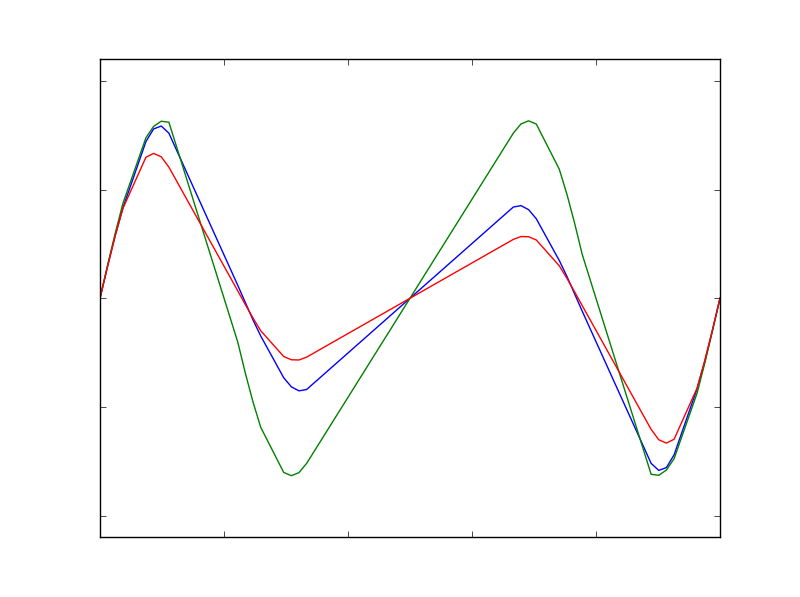}} 
{\includegraphics[width=0.45\textwidth]{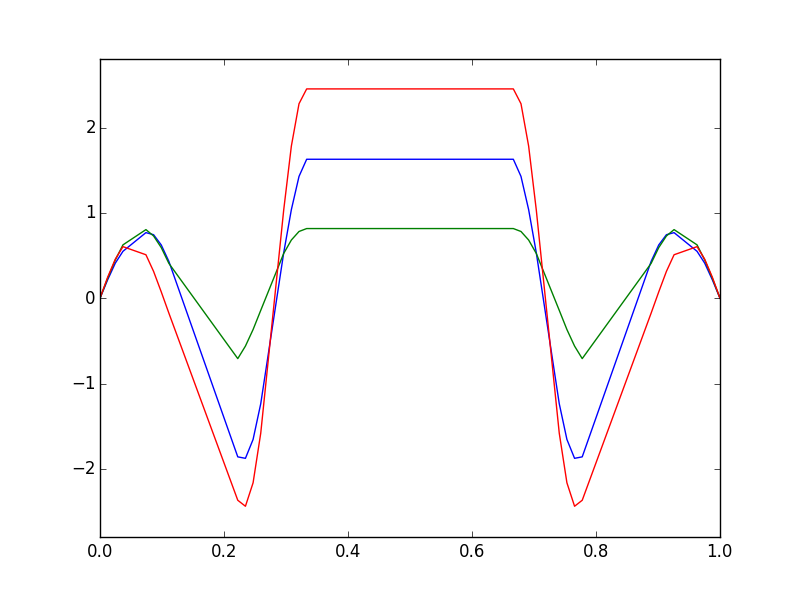}}
{\includegraphics[width=0.45\textwidth]{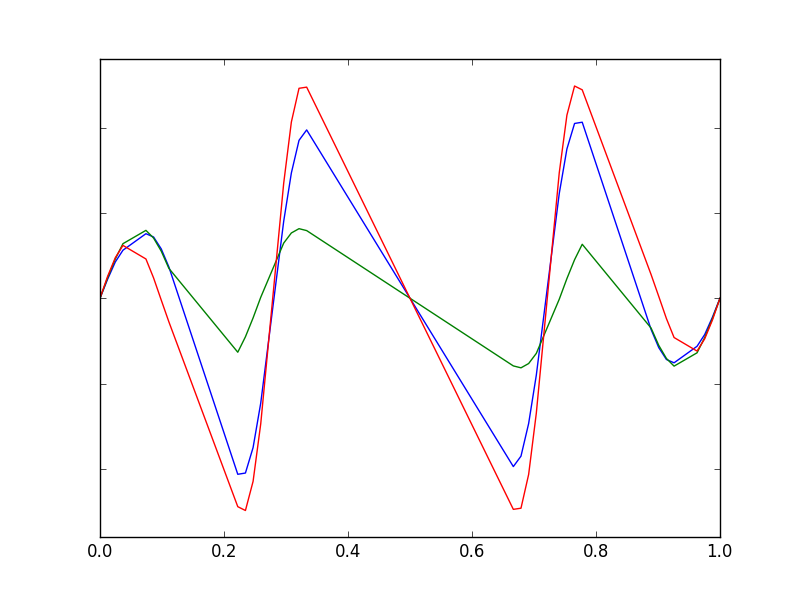}}
 \caption{Dirichlet eigenfunctions}
 \label{fig10}
 \end{center} 
\end{figure}	
\vspace{40pt}

 {\large \textbf{Acknowledgement}}\\[5pt]
The authors thank the anonymous referee for helpful suggestions for improvement. In particular, thanks to some simple changes and rearrangements, we could formulate and prove our results in much more generality.

\end{document}